\documentclass[11pt,a4paper]{amsart}
\usepackage{amssymb,amsmath,amsthm}
\usepackage{inputenc}
\usepackage{graphicx}
\usepackage{subcaption}
\usepackage{units}

\newtheorem{theorem}{Theorem}[]
\newtheorem{corollary}{Corollary}
\newtheorem{lemma}{Lemma}

\newtheorem{claim}{Claim}

\newtheorem{problem}{Problem}

\renewcommand{\geq}{\geqslant}
\renewcommand{\leq}{\leqslant}

\newcommand{\tr}{\mathrm{tr}}

\newcommand{\R}{\mathbb{R}}

\newcommand{\eps}{\varepsilon}

\newcommand{\la}{\langle}
\newcommand{\ra}{\rangle}

\newcommand{\wv}{\widetilde{v}}
\newcommand{\wm}{\widetilde{m}}

\newcommand{\M}{\mathcal{M}}

\begin{document}

\title{Uniform tight frames as optimal signals}

\author{Gergely Ambrus, Bo Bai,  Jianfeng Hou}
\address{Gergely Ambrus, Alfr\'ed R\'enyi Institute of Mathematics, E\"{o}tv\"{o}s Lor\'{a}nd Research Network,
POB 127 H-1364 Budapest, Hungary.}
\email{ambrus@renyi.hu}
\thanks{Research of the first author was supported by NKFIH grants PD-125502 and KKP-133819.}

\address{Bo Bai, Theory Lab HK, Huawei Tech. Investment Co., Limited, Shatin, N. T., Hong Kong.}
\email{baibo8@huawei.com}

\address{Jianfeng Hou, Theory Lab HK, Huawei Tech. Investment Co., Limited, Shatin, N. T., Hong Kong; Center of Discrete Mathematics, Fuzhou University, Fujian, China, 350116.}
\email{hou.jianfeng@huawei.com}



\maketitle

\begin{abstract}
Non-orthogonal communication is a promising technique for future wireless networks (e.g., 6G and Wi-Fi 7). In the vector channel model, designing efficient non-orthogonal communication schemes amounts to the following extremum problem:
\[
\max \min_k \frac{|v_k|^2}{\sigma^2 + \sum_{l \neq k} \la v_k, v_l \ra^2}
\]
where the maximum is taken among vector systems $(v_k)_1^N \subset \R^d$ satisfying $c_1 \leq |v_k|^2 \leq c_2$ for every $k$, and the parameter $\sigma>0$ corresponds to the noise of the channel. We show that in the case $\sigma = 0$,  uniform tight frames are the only optimal configurations. We also give quantitative bounds on the optimal capacity of vector channels with relatively small noise.
\end{abstract}

\section{Introduction}\label{introd}

As essential elements in wireless communications, orthogonal frequency division multiplexing (OFDM) and multiple-input multiple-output (MIMO) have been widely deployed in cellular communications (e.g., 4G and 5G) and Wi-Fi networks. In an  OFDM system, the transmitter and receiver (for example, base station, access point, smart phones and other user devices) uses multiple orthogonal subcarriers to transmit information. In a typical MIMO system, on the other hand, the transmitter and receiver are usually equipped with multiple antennas to enhance information transmission efficiency and reliability \cite{C13, L18}.

To further improve transmission efficiency, non-orthogonal communication schemes attract much attention  both from academia and industry. In this work, we study the non-orthogonal communication problem in the following simple yet essential vector channel model \cite{G08}:
\begin{equation}\label{channel-model}
  y=\sum_{k=1}^Nv_k+w,
\end{equation}
where $v_k\in\R^d$ represents the vector sent by the $k$-th transmitter, and $y\in\R^d$ is the vector received at the receiver, with $w\in\mathbb{R}^d$ being the Gaussian noise vector with $w\sim N(0,\sigma^2)$. As we are approaching the massive machine type communications beyond 5G, it is highly required to support huge number of low rate users with limited channel dimensions $d$. Specifically, the aim is to find the proper communication scheme (i.e., construct $v_k$ with $k=1,\ldots,N$ for $N\geq d$) so that the optimal channel capacity (i.e., the maximum rate of reliable communication) may be achieved.

For the purpose of error correction, we may choose the vectors $v_i$ so that the distance of any two of them is as large as possible.  This is closely related to the {\em spherical coding} (or {\em packing}) problem, in which the goal is to find a set of $N$ points (codewords) on the unit sphere $S^{d-1}$ of $\R^d$ so that the minimum distance between the $N$ points is as large as possible~\cite{B04, C93}. The spherical code method has gained popularity in connection with the construction of spreading sequences for Code-Division Multiple-Access (CDMA) systems~\cite{V99}.

In the present article, we set off to maximize channel capacity. This is defined to be the theoretical smallest upper bound on the information rate of data that can be communicated at an arbitrarily low error rate using an average received signal power $S$ through an analog communication channel subject to additive white Gaussian noise (AWGN) power $\Phi$, where the unit is $\unit{bits/symbol}$. The classical Shannon-Hartley Theorem \cite{H28, S49} states that the channel capacity $C$ is given by
\[
C=\frac{1}{2}\log_2\left(1+\frac{S}{\Phi}\right).
\]

Returning to the vector channel model \eqref{channel-model}, if the $k$-th transmitter is assigned to a codeword $v_k \in \R^d$, then its received signal power $S$ is $|v_k|^2$, the squared Euclidean norm of $v_k$ (see \cite{G08}). The noise of $k$-th transmitter consists of two parts: one originates of its own channel noise $w\sim N(0,\sigma^2)$, while the other part is yielded by the interference with the other transmitters, which is expressed by the quantity $\sum_{l \neq k} \la v_k, v_l \ra^2$. Therefore, the power of noise may be expressed as $\Phi=\sigma^2+\sum_{l \neq k} \la v_k, v_l \ra^2$  (see \cite{K08,T99}), and the channel capacity from the $k$-th transmitter to the receiver can be formulated as
\begin{equation*}\label{chanel-capacity}
\log \left( 1 + \frac{|v_k|^2}{\sigma^2 + \sum_{l \neq k} \la v_k, v_l \ra^2} \right),
\end{equation*}
where $\frac{1}{2}$ is omitted in the equation (throughout the article we use the convention $1 / 0 = \infty$.) The aim is to find $v_1, \ldots, v_N \in \R^d$ such that the minimal channel capacity is as large as possible. Accordingly, we address the following question:

\begin{problem}\label{problem1}
Assume that $d \geq 2$, $N \geq d$, $0 < c_1 < c_2$ are positive bounds, and $\sigma \geq 0$ is a constant. Determine the quantity
\begin{equation}\label{maxmin}
\max_{\substack{{v_1, \ldots, v_N \in \R^d} \\ {c_1 \leq |v_i|^2 \leq c_2 \ \forall i}}} \min_{1 \leq k \leq N}  \log \left( 1 + \frac{|v_k|^2}{\sigma^2 + \sum_{l \neq k} \la v_k, v_l \ra^2} \right) \,.
\end{equation}
\end{problem}

In the present article, we solve Problem~\ref{problem1} in the special case $\sigma =0$, and obtain a stability version for small values of $\sigma>0$. The latter is essential for practical applications in signal processing.

We start by a trivial simplification.  Note that for any strictly monotone increasing function $f$, the maxima of
   \[
   \min_{1 \leq k \leq N} f \left( \frac{|v_k|^2}{\sigma^2 + \sum_{l \neq k} \la v_k, v_l \ra^2} \right)
   \]
and
\[
\min_{1 \leq k \leq N}  \frac{|v_k|^2}{\sigma^2 + \sum_{l \neq k} \la v_k, v_l \ra^2}
\]
are attained at the same vector configurations (subject to arbitrary boundary conditions). Since $ \log(1 + x)$ is strictly monotone increasing on $[0, \infty)$, we may consider the latter target function when seeking the solution of Problem~\ref{problem1}.

In order to formulate our results, we introduce a couple of notions. We are going to call a vector system $(v_i)_1^N \subset \R^d$  {\em uniform} if $|v_i| =c$ holds for every $i$ with some constant $c > 0$. Equivalently, $(v_i)_1^N \subset c S^{d-1}$, where $S^{d-1}$ denotes the unit sphere in $\R^d$. The vector system $v_1, \ldots, v_N$ is a {\em uniform tight frame} of norm $c$, if $|v_k| = c$ holds for every $k \in [N]$ (where $[N] = \{ 1, \ldots, N \}$), and
\begin{equation*}\label{vi_frame}
\sum_{i=1}^N v_k \otimes v_k = \frac{N c^2}{d}I_d.
\end{equation*}
Some basic properties of tight frames are collected in the subsequent section.

First, we study the $\sigma=0$ case, that is, when the channel is assumed to be noise-free. According to the above remarks, our task is to find the vector systems maximizing
\begin{equation}\label{Mdef}
M(v_1, \ldots, v_N) = \min_{1 \leq k \leq N}  \frac{|v_k|^2}{ \sum_{l \neq k} \la v_k, v_l \ra^2}
\end{equation}
subject to $c_1 \leq |v_i|^2 \leq c_2$ for every $1 \leq i \leq N$. We are going to call vector systems for which the maximum is attained to be {\em extremal}.

When $N \leq d$, \eqref{Mdef} is maximized when $(v_i)_1^N$ is an orthogonal system. In this case, the denominator is $0$ for every $k$, thus, $M(v_1, \ldots, v_N) = \infty$. Clearly, only orthogonal systems correspond to this value. Thus, from now on we may assume that the number of the vectors exceeds $d$, hence, $M(v_1, \ldots, v_N) < \infty$.

\begin{theorem}\label{thm1} Assume that $2 \leq d < N$, and $0 < c_1 < c_2$.
The vector system $v_1, \ldots, v_N \subset \R^d$ is a maximizer of $M(v_1, \ldots, v_N)$ defined in \eqref{Mdef} subject to the condition $c_1 \leq |v_i|^2 \leq c_2$ for every  $ i \in [N]$ if and only if $(v_i)_1^N$ is a uniform tight frame of norm $\sqrt{c_1}$.
\end{theorem}

By a simple calculation (see \eqref{extrvalue}) we obtain the optimal estimate for the capacity of a noise-free channel.

\begin{corollary}\label{cor1}
   The answer to Problem~\ref{problem1} when $\sigma = 0$  is
\[
\max_{\substack{{v_1, \ldots, v_N \in \R^d} \\ {c_1 \leq |v_i|^2 \leq c_2 \ \forall i}}} \min_{1 \leq k \leq N} N \log \left( 1 + \frac{|v_k|^2}{\sum_{l \neq k} \la v_k, v_l \ra^2} \right) = N  \log \left( 1+ \frac {d}{c_1 (N - d)}\right) \,.
\]
\end{corollary}

We note that the answer to Problem~\ref{problem1} clearly depends on the value of $\sigma$: not only the optimal capacity does so, but the structure of the extremal vector systems as well. To illustrate this, assume that $\sigma$ is very large compared to $c_2 N$. In this case, the dominant term of $\sigma^2 + \sum_{l \neq k} \la v_k, v_l \ra ^2$ is the first one. Therefore, the extremum of \eqref{maxmin} is attained when $ | v_i|^2 = c_2 $ for every $i$ -- that is, the vector norms are {\em maximal}, as opposed to the case $\sigma = 0$.

However, in practical applications, we may assume that the noise is relatively small. This is the situation that we are going to study.  First, we restrict the search to uniform vector systems.

\begin{theorem}\label{thm2}
Assume that $\sigma \leq c_1 \sqrt{(N-d)/d}$. Then there is a uniform tight frame of norm $\sqrt{c_1}$ which maximizes
\begin{equation}\label{maxmin_unif}
   \min_{1 \leq k \leq N}  \frac{|v_k|^2}{\sigma^2 + \sum_{l \neq k} \la v_k, v_l \ra^2}
\end{equation}
among uniform vector systems $(v_i)_1^N \subset \sqrt{c} S^{d-1}$ with  $c_1 \leq c \leq c_2$.
\end{theorem}

Calculating the corresponding channel capacity (see \eqref{cor2value}) yields:

\begin{corollary}\label{cor2}
Assuming that $\sigma \leq c_1 \sqrt{(N-d)/d}$ and that $(v_i)_1^N$ is a uniform vector system, 
\begin{equation}\label{maxminunifanswer}
  \begin{split}
  \max_{\substack{{v_1, \ldots, v_N \in \sqrt{c} S^{d-1}}\\{ c_1 \leq c \leq c_2}}}  \min_{1 \leq k \leq N} N &\log \left( 1 + \frac{|v_k|^2}{\sigma^2 + \sum_{l \neq k} \la v_k, v_l \ra^2}  \right)\\ &= N  \log  \left( 1 + \frac {c_1}{\sigma^2 + c_1 ^2 (N-d) / d} \right) \,.
\qedhere
\end{split}
\end{equation}
\end{corollary}

Next, we consider the general case. Although extremal vector systems are not necessarily uniform, we show that for small $\sigma$, there exists an extremal vector system containing relatively few vectors of non-minimal norm.

\begin{theorem}\label{thm3}
  Assuming that $\sigma<c_1 / \sqrt{d}$, there exists a vector system which is extremal with respect to Problem~\ref{problem1} containing at most
  \[
  d \frac {c_1^2 - \sigma^2}{c_1^2 - d \sigma^2 }
  \]
  vectors of norm strictly larger than $\sqrt{c_1}$.
\end{theorem}

For channels with a larger amount of noise, we provide the following bound on the number vectors of non-minimal norm (note that this is indeed weaker for large values of $c_2/c_1$).

\begin{theorem}\label{thm4}
For all $\sigma \geq 0$, there exists a vector system which is extremal with respect to Problem~\ref{problem1} containing at most
\begin{equation}\label{thm4est}
  d \,\frac { (2 \sigma^2 + 2 c_2^2 - c_1^2)}{c_1^2 }
\end{equation}
vectors of norm strictly larger than $\sqrt{c_1}$.
\end{theorem}

We conclude the article by proving the following stability estimate for the channel capacity in the general case under the assumptions that the noise of the channel is not too large, and the number of vectors is sufficiently large.

\begin{theorem}\label{thm5} Assume that $\sigma<c_1 / \sqrt{2d }$ and that $N > 2 d c_2^2 / c_1^2$. Then
\begin{equation*}\label{eq_thm5}
  \begin{split}
  \max_{\substack{{v_1, \ldots, v_N \in \R^d} \\ {c_1 \leq |v_i|^2 \leq c_2 \ \forall i}}} \min_{1 \leq k \leq N}  \log \left( 1 + \frac{|v_k|^2}{\sigma^2 + \sum_{l \neq k} \la v_k, v_l \ra^2} \right) \\
  \leq \log \left( 1 +  \frac{d}{c_1(N -d) + \sigma^2 \cdot \frac{d}{c_1}   - \frac {d^2}{c_1 N}  (c_2^2 - c_1^2) \frac{c_1^2 - \sigma^2 } { c_1^2 - d \sigma^2 }} \right)\,.
\end{split}
\end{equation*}
\end{theorem}

This provides a quantitative estimate on the difference between \eqref{maxmin} and \eqref{maxminunifanswer}, showing that for practical applications, using a uniform tight frame of norm $\sqrt{c_1}$ as the set of possible codewords is a well-justified choice.

\section{Tight frames}

From the theoretical viewpoint, Problem~\ref{problem1} is closely related to the notion of {\em frames}, introduced originally by Duffin and Schaeffer \cite{DS52}. A vector system $(v_i)_1^N \subset \R^d$ is called a {\em frame} if there exist $0 < A \leq B < \infty $ such that
\[
A |w|^2 \leq \sum_{i = 1}^N \la w, v_i \ra^2 \leq B |w|^2
\]
holds for every vector $w \in \R^d$. If $A = B$ holds above, the vector system is a tight frame. Frame theory has become a well-studied topic in recent years, with plenty of real-world applications. Of the excessive literature on frame theory and its application in information theory, we only pick the volumes \cite{CK13} and \cite{O16}, in which the interested reader may find ample literature on the subject.

An alternative definition of tight frames involves the notion of the {\em tensor product} of the vectors $u,v \in \R^d$, which is the $\R^d \rightarrow \R^d$ linear map $u \otimes v$  satisfying
\[
(u \otimes v) z =  u \la z,v \ra
\]
for every $z \in \R^d$.
Given a vector system $(v_i)_1^N \subset \R^d$, we define its {\em frame operator \cite{BF03}} $A$ by
\begin{equation}\label{Def_frameoperator}
  A(v_1, \ldots, v_N) = \sum_{i=1}^N v_i \otimes v_i \,.
\end{equation}

A set of vectors $v_1, \ldots, v_N$ in $\R^d$ is called a {\em tight frame} if its frame operator is a constant multiple of the identity operator, that is,
\begin{equation}\label{fra}
\sum_{i=1}^N v_i \otimes v_i = \lambda \, I_d
\end{equation}
with a real constant $\lambda \in \R$. This is equivalent to requiring that
\[
\sum_{i=1}^N \la w, v_i \ra^2 = \lambda |w|^2
\]
holds for every vector $w \in \R^d$.

A uniform vector system $(v_i)_1^N \subset \R^d$ which satisfies \eqref{fra} is called a {\em uniform tight frame}. In the special case when the common norm is 1, we talk about a {\em unit norm tight frame} (UNTF). By comparing traces in \eqref{fra}, it immediately follows that in this latter case, $\lambda = N/d$. The complete characterization of unit norm tight frames was given by Benedetto and Fickus \cite{BF03} -- it also follows that UNTF's exist for every $N \geq d$ (see \cite{GVT98} as well, and \cite{I20} for the non-uniform case).

We associate to a vector system $(v_i)_1^N \subset \R^d$ its {\em frame potential } (or 2-frame potential \cite{EO12}) defined by
\begin{equation*}\label{framepot}
FP\left(v_1, \ldots, v_N\right) = \sum_{i,j = 1} ^N \la v_i, v_j \ra ^2.
\end{equation*}
The frame potential was introduced by Duffin and Schaeffer~\cite{DS52} (see \cite{BF03}, \cite{FJKO05} and \cite{CGGKO20} for further applications and generalizations).

Let $G(v_1, \ldots, v_N)$ denote the Gram matrix corresponding to the vector system $(v_i)_1^N$, that is, the $N \times N$ matrix $G$ satisfying
\begin{equation*}\label{gramdef}
  G(v_1, \ldots, v_N)_{ij} = \la v_i , v_j \ra \,.
\end{equation*}
If $L$ denotes the $N \times d$ matrix with rows $v_1^\top,\ldots, v_N^\top $, then
\begin{equation}\label{Gram2}
G(v_1, \ldots, v_N) = L L^\top,
\end{equation}
and on the other hand,
\begin{equation}\label{FrameOp2}
A(v_1, \ldots, v_N) = L^\top L.
\end{equation}

The frame potential of the vector system may be expressed as
\[
FP (v_1, \ldots, v_N ) =  \tr \, G^2 = \sum_{i,j = 1}^N G_{ij}^2 =  \| G \|_{HS}^2,
\]
the square of the Hilbert-Schmidt norm of $G$. Thus, using \eqref{Gram2}, \eqref{FrameOp2}, and the property that for arbitrary $N \times N$ matrices $R,S$, $\tr (R S^\top) = \tr (R^\top S)$,
\begin{equation}\label{FPduality}
FP\left((v_i)_1^N\right) = \| G \| _{HS}^2 =  \tr (L L^\top L L^\top) = \tr (L^\top L L^\top L) =  \| A \|_{HS} ^2.
\end{equation}
The above formula is called the {\em frame potential duality,} which lies at the core of the proof of the characterization of UNTF's \cite{BF03}.

\section{The noise-free case}

\begin{proof}[Proof of Theorem \ref{thm1}]
Let $(v_1, \ldots, v_N)$ be an extremal vector system, and introduce
\begin{equation}\label{mkdef}
  m_k = \frac{|v_k|^2}{ \sum_{l \neq k} \la v_k, v_l \ra^2}
\end{equation}
for every $k = 1, \ldots, N$.
Then, by \eqref{Mdef}, $M(v_1, \ldots, v_N) = \min_k m_k$, and since $(v_i)_1^N$ is extremal, $M(v_1, \ldots, v_N)$ is maximal among the suitable vector systems. Call a direction vector $u \in S^{d-1}$ {\em minimal}, if $u = v_k / |v_k|$ for some $k \in [N]$ with $m_k = M(v_1, \ldots, v_N)$. Denote by $\M(v_1, \ldots, v_N)$ the set of minimal directions corresponding to the vector system $(v_i)_1^N$.

We will show that extremal vector systems are uniform. To that end, assume $|v_i| > \sqrt{c_1}$ for some $i \in [N]$. We alter the vector system by defining
\[
\widetilde{v}_k = \begin{cases}
v_k, & \text{for }k \neq i\\
\frac{\sqrt{c_1}}{|v_k|} v_k, & \text{for }k = i
\end{cases}
\]
for every $k\in [N]$. Accordingly, introduce
\begin{equation}\label{mk2def}
  \widetilde{m}_k = \frac{|\widetilde{v}_k|^2}{ \sum_{l \neq k} \la \widetilde{v}_k, \widetilde{v}_l \ra^2}
\end{equation}
for every $k \in [N]$.
\begin{claim}\label{claim1}
  The vector system $(\widetilde{v}_1,  \ldots, \widetilde{v}_N)$ described above is also extremal. Moreover, $\M(\widetilde{v}_1,  \ldots, \widetilde{v}_N) \subseteq \M(v_1, \ldots, v_N)$ holds, with equality if and only if $v_i$ is orthogonal to every direction in $\M(v_1, \ldots, v_N)$ different from $v_i / |v_i|$.
\end{claim}

\begin{proof}
  Clearly, $m_i = \wm_i$. Taking any $k \in [N] \setminus \{i\}$, we have that $\la \widetilde{v}_k, \wv_i \ra^2  \leq \la v_k , v_i \ra^2$, where equality holds if and only if  $\la v_k , v_i \ra = 0$. Since all the terms $\la v_k, v_l \ra$ not involving $v_i$ remain unchanged, we see that $\wm_k \geq m_k$ for every $k \in [N]$. In particular, $\min_{k \in [N]}\wm_k \geq \min_{k \in [N]} m_k$, and since this latter is globally maximal, we derive that $(\wv_1, \ldots, \wv_N)$ must be extremal too.

  For the second statement, the inclusion is trivial by the above argument. Notice that $m_k = \wm_k$ holds if and only if $k = i$ or $\la v_k , v_i \ra = 0$. Thus, if  $\M(\widetilde{v}_1,  \ldots, \widetilde{v}_N) = \M(v_1, \ldots, v_N)$, then every minimal direction is either the direction of $v_i$, or orthogonal to it.
\end{proof}

Applying Claim~\ref{claim1} repeatedly to each vector of norm greater than $\sqrt{c_1}$ leads to a uniform vector system of norm $\sqrt{c_1}$ which is extremal. By scaling, we may assume that $c_1 = 1$. Next, we characterize uniform extremal vector systems using an argument along the lines of Theorem 6.2. in \cite{BF03}.

Clearly,
\[
 \min_{1 \leq k \leq N}  \frac{1}{ \sum_{l \neq k} \la v_k, v_l \ra^2}
\]
is maximized if and only if its reciprocal is minimized. Thus, we may study the extremum problem
\begin{equation*}\label{Mrec}
  \min_{v_1, \ldots, v_N \in S^{d-1}} \max_{1 \leq k \leq N} \sum_{l \neq k} \la v_k, v_l \ra^2.
\end{equation*}
Since $|v_k| = 1$ for every $k \in [N]$, this is attained at the same configurations as the minmax of
\[
E_k := \sum_{l=1}^N \la v_k, v_l \ra^2.
\]
By frame potential duality \eqref{FPduality},
\begin{align*}\label{}
  N \max_{1 \leq k \leq N} E_k &\geq {\sum_{k, l =1} ^N \la v_k, v_l \ra^2 }  \\
  &=  \| G (v_1, \ldots, v_N) \| _{HS}^2 \\
  &= \| A(v_1, \ldots, v_N)\|_{HS}^2 \,.
\end{align*}
Since \eqref{Def_frameoperator} shows that $\tr A = N$, the Cauchy-Schwarz inequality applied to the diagonal entries of $A$ implies that
\begin{equation}\label{AHS}
\| A(v_1, \ldots, v_N)\|_{HS}^2  \geq \frac{N^2}{d}\,,
\end{equation}
therefore,
\[
\min_{v_1, \ldots, v_N \in S^{d-1}} \max_k  \sum_{l=1}^N \la v_k, v_l \ra^2 \geq \frac{N}{d} \,.
\]
Note that by \eqref{Def_frameoperator}, diagonal entries of $A(v_1, \ldots, v_N)$ are non-negative. Thus,  equality may hold in \eqref{AHS} only if all diagonal entries of $A(v_1, \ldots, v_N)$ are equal, and all off-diagonal entries are 0. Therefore, $A = \frac N d I_d$, that is, the vectors $v_i$ form a UNTF.  In this case, the above bounds are indeed achieved.

This completes the characterization of uniform extremal systems: these are uniform tight frames of norm $\sqrt{c_1}$. Then,
\begin{equation}\label{extrvalue}
  \frac{|v_k|^2}{ \sum_{l \neq k} \la v_k, v_l \ra^2} = \frac {d}{c_1 (N-d)}
\end{equation}
holds for every $1 \leq k \leq N$. Thus, $v_k / |v_k|$ is a minimal direction for every $k \in [N]$.

Let us return to the general case. Let $(v_i)_1^N$ be an extremal vector system.
 Claim~\ref{claim1} implies that $\M(v_1, \ldots, v_N)$ contains the direction of every vector $v_k$,  which is only possible if each
 vector of norm exceeding $\sqrt{c_1}$ is orthogonal to all the other vectors. Thus, the system $(v_i)_1^N$ must be the union of an orthogonal base of an $r$-dimensional subspace $H$ consisting of vectors of norm  in $(\sqrt{c_1}, \sqrt{c_2}]$, and a $\sqrt{c_1}$-norm tight frame of $S^\perp$ consisting of $N-r$ vectors. However, in this case, the value of \eqref{Mdef} is
\[
M(v_1, \ldots, v_N) =  \frac{d-r}{c_1 (N - d)}
\]
by \eqref{extrvalue}. This shows that the vector system may only be extremal when $r=0$, that is, the vector system is a uniform tight frame.
\end{proof}

\section{Results for $\sigma^2 >0$}

\begin{proof}[Proof of Theorem~\ref{thm2}]
Let $|v_i|^2 =  c$ for every $i$ with $c \in [c_1, c_2]$. Clearly,  maximizing \eqref{maxmin_unif} on $\sqrt{c} S^{d-1} $ is equivalent to solving
\begin{equation}\label{mmm}
\min_{v_1, \ldots, v_N \in \sqrt{c}\, S^{d-1}}  \max_{1 \leq k \leq N} \frac{\sigma^2 + \sum_{l \neq k} \la v_k, v_l \ra^2}{c} \,.
\end{equation}
For a fixed value of $c$, the contribution of the term $\sigma^2 /c$ is constant, therefore it may be omitted from the target function, and the results of the previous section apply. Therefore, the extremum value is attained when the vector system is a uniform tight frame of norm $\sqrt{c}$, and the extremal value of \eqref{mmm} is
\begin{equation}\label{cfunction}
\frac{\sigma^2}{c} + c \,\frac {N -d}{d}\,.
\end{equation}
Thus, we need to minimize the above quantity as a function of $c$ over the interval $[c_1, c_2]$. Since $N>d$,  \eqref{cfunction} is decreasing on the interval $[0, \sigma \sqrt{d / (N- d)}]$ and is increasing for $c >  \sigma \sqrt{d / (N- d)}$.
 Thus, when $c_1 \geq \sigma \sqrt{d / (N- d)}$, the minimum over the interval $[c_1, c_2]$ is attained at $c = c_1$.
\end{proof}

In the extremal case, by \eqref{mmm} and \eqref{cfunction},
\begin{equation}\label{cor2value}
  \frac{|v_k|^2}{\sigma^2 + \sum_{l \neq k} \la v_k, v_l \ra^2} = \frac{c_1}{\sigma^2 + c_1^2 (N - d) / d}  \, ,
\end{equation}
which proves Corollary~\ref{cor2}.

\begin{proof}[Proof of Theorem~\ref{thm3}]
Let $(v_i)_1^N$ be a vector system satisfying the boundary conditions $c_1 \leq |v_i|^2 \leq c_2$ for every $i$.

Let $I \subset [N]$ be a subset of indices with $|I| \geq 2$ so that $|v_i|^2 > c_1$ for every $i \in I$ (we will specify $I$ later). Introduce the {\em simultaneous scaling} corresponding to $I$ by a factor $\lambda <1 $ of  $(v_i)_1^N$ by setting
\[
\widetilde{v_i} = v_i
\]
for $i \not \in I$, and
\[
\widetilde{v_i} =  \lambda v_i
\]
for $i \in I$. If $\lambda<1$ is close enough to 1, all vectors of the simultaneously scaled configuration have norm between $\sqrt{c_1}$ and $\sqrt{c_2}$.

As in \eqref{mkdef} and \eqref{mk2def}, let
\[
\mu_k = \frac{|v_k|^2}{\sigma^2 + \sum_{l \neq k} \la v_k, v_l \ra^2}\
\]
and
\[
\widetilde{\mu}_k = \frac{|\widetilde{v}_k|^2}{\sigma^2 + \sum_{l \neq k} \la \widetilde{v}_k, \widetilde{v}_l \ra^2}\,.
\]
We study the effect of  simultaneous scaling on the values $\mu_k$:
\begin{claim}\label{claim2}
  Assume that for the index set $ I \subset [N]$ consisting of at least 2 indices, \begin{equation}\label{signeq2}
\sigma^2 < \sum_{l \in I\setminus \{k \}} \la v_k, v_l \ra^2
\end{equation}
holds for every $k \in I$. Then for sufficiently small values of $\eps>0$,  the simultaneous scaling corresponding to $I$ with factor $\lambda = 1- \eps$ does not decrease any of the terms $\mu_k$. That is, $\widetilde{\mu}_k \geq \mu_k$ holds for every $k \in [N]$. In particular, if $(v_i)_1^N$ is extremal, then $(\widetilde{v}_i)_1^N$ needs to be extremal as well.

\end{claim}

\begin{proof}
 If $k \not \in I$, then $|v_k|$ is unchanged, while the denominator does not increase (it decreases if and only if there is $i \in I$ with $\la v_i,v_k \ra \neq 0$).  Thus,
\[
\widetilde{\mu}_k \geq \mu_k
\]
for every $k \not \in I$.

Assume now that $k \in I$. Then,
\[
\widetilde{\mu}_k  = \frac { \lambda^2 |v_k|^2}{\sigma^2 + \lambda^2 \sum_{l \not \in I} \la v_k, v_l \ra^2 + \lambda^4 \sum_{l \in I \setminus \{ k\}} \la v_k, v_l \ra^2} \, .
\]
Calculating the derivative of $\widetilde{\mu}_k $ with respect to $\lambda$ at $\lambda = 1$, one obtains  that its sign agrees to that of
\begin{equation}\label{signeq}
  \sigma^2 - \sum_{l \in I \setminus \{ k\}} \la v_k, v_l \ra^2 \,.
\end{equation}
Therefore, \eqref{signeq2} implies that the derivative is strictly negative for every $k \in I$, which suffices for the proof.
\end{proof}

Let now $(v_i)_1^N$ be an extremal vector system which, among the extremal configurations, minimizes $\sum_{i=1}^N |v_i|^2$. Denote by $M$ the number of vectors of norm strictly larger than $\sqrt{c_1}$ -- we may and do assume that $M \geq d$ and  these vectors are $v_1, \ldots, v_M$. The following classical bound guarantees the existence of two of these vectors whose inner product is large in absolute value.

\begin{lemma}[Welch \cite{W74}]\label{lemma_welch}
Assume that $M$ vectors $w_1, \ldots, w_M \subset \R^d$ are given so that $|w_i|^2 \geq c_1$ for every $i$. Then
\begin{equation}\label{eq_welch}
  \max_{ i \neq j} \la w_i, w_j \ra ^2 \geq \frac{c_1^2 (M-d)}{d (M-1)} \,.
\end{equation}
\end{lemma}

We note that an alternative bound has recently been proven by Bukh and Cox \cite{BC19}, which is stronger for $M \approx d + \sqrt{d}$. Yet, for our needs, the above estimate is sufficient.

Let now $i,j \in [M]$ be the indices provided by Lemma~\ref{lemma_welch}, and set $I = \{i,j\}$. Perform the simultaneous scaling corresponding to the index set $I$ with some factor $\lambda <1$. Due to the minimality of $\sum_{i=1}^N |v_i|^2$, the scaled vector system may not be extremal. Therefore, the condition of Claim~\ref{claim2} must be violated:
\[
\sigma^2 \geq \la v_i, v_j \ra^2.
\]
Thus, by \eqref{eq_welch},
\[
\sigma^2 \geq \frac{c_1^2 (M-d)}{d (M-1)}\,.
\]
Rearranging for $M$ we derive
\begin{equation}\label{Mineq}
  M < d \frac {c_1^2 - \sigma^2}{c_1^2 - d \sigma^2 }
\end{equation}
provided that $c_1^2 - d \sigma^2 >0$ holds.
\end{proof}

\begin{proof}[Proof of Theorem~\ref{thm4}]
Instead of Lemma~\ref{lemma_welch}, we now apply

\begin{lemma}\label{lemma2}
  Assume that $Q$ is an $M \times M$ symmetric matrix with nonnegative entries. Then there exists an index set $J \subset [M]$ such that for every $k \in J$,
  \begin{equation}\label{eq_lem2}
    \sum_{l \in J} Q_{kl} \geq  \frac{\sum_{i,j=1}^M Q_{ij}}{2 M} + \frac{Q_{kk}}{2}.
  \end{equation}
  
\end{lemma}

\begin{proof}
  Suppose on the contrary that the above inequality is not true. Starting with $[M]$, remove the indices one-by-one, selecting in each step the index $k$ of a row with minimal sum of the principal minor corresponding to the current index set. Removing this index results in deleting the corresponding row and column from the minor. By the above assumption, the sum of the entries removed is strictly less than
  \[
  2 \left(\frac{\sum_{i,j=1}^M Q_{ij}}{2 M} + \frac{Q_{kk}}{2} \right)  - Q_{kk} = \frac{\sum_{i,j=1}^M Q_{ij}}{ M} \,.
  \]
  Since this holds for every step, the sum of all the entries removed during the $M$ steps of the process is strictly less than $\sum_{i,j=1}^M Q_{ij}$, which contradicts to the fact that we remove all entries of $Q$.
\end{proof}

As before, let $(v_i)_1^N$ be an extremal vector system with minimal $\sum_{i=1}^N |v_i|^2$, and assume that the vectors which have norm $>\sqrt{c_1}$ are exactly $v_1, \ldots, v_M$. Our goal is to show that \eqref{thm4est} holds. Assume on the contrary that
\begin{equation}\label{Mlow}
 M > d \,\frac { (2 \sigma^2 + 2 c_2^2 - c_1^2)}{c_1^2 }\,.
\end{equation}

Let $Q$ be the $ M \times M$ matrix defined by $Q_{i,j} = \la v_i, v_j \ra^2$. By \eqref{FPduality} and the Cauchy-Schwarz inequality,
\begin{equation}\label{qijnorm}
\begin{split}
\sum_{i,j,=1}^M Q_{i,j}&=\left\|\sum_{i=1}^M v_i \otimes v_i \right\|_{HS}^2 \\ &\geq d \left( \frac {\sum_{i=1}^M |v_i|^2}{d} \right)^2 >  d \left( \frac {M c_1}{d} \right)^2 = \frac{M^2 c_1^2}{d}\,.
\end{split}
\end{equation}
Thus, Lemma~\ref{lemma2} implies that we may select a  set of indices $J \subset [M]$ for which
\begin{equation}\label{qij}
\sum_{l \in J} Q_{kl} \geq \frac{\sum_{i,j=1}^M Q_{ij}}{2 M} + \frac{Q_{kk}}{2} > \frac{M c_1^2}{2 d} + \frac{c_1^2}{2}
\end{equation}
holds for every $k \in J$. 

Next, we show that $J$ may not be a singleton. Indeed, suppose that $J = \{k\}$. Then, by \eqref{eq_lem2} and \eqref{qijnorm},
\[
Q_{kk} \geq  \frac{\sum_{i,j=1}^M Q_{ij}}{ M} > \frac{M c_1^2}{d}\,.
\]
On the other hand, $Q_{kk} = |v_k|^4 \leq c_2^2$. This implies that $M < d c_2^2 / c_1^2$, which contradicts \eqref{Mlow}.

Thus, we may assume that $|J| \geq 2$. By \eqref{qij}, for all $k \in J$,
\begin{equation}\label{eq_j}
  \sum_{l \in J\setminus \{k \}} \la v_k, v_l \ra^2 > \frac{M c_1^2}{2 d} + \frac{c_1^2}{2} - c_2^2.
\end{equation}
Note that \eqref{Mlow} implies that 
\[
\sigma^2 < \frac{M c_1^2}{2 d} + \frac{c_1^2}{2} - c_2^2.
\]
Thus, by \eqref{eq_j}, the conditions of Claim~\ref{claim2} are satisfied. Hence, the simultaneous scaling corresponding the index set $J$ and factor $1 - \eps$ for sufficiently small $\eps$ yields another extremal vector system.  This contradicts to the minimality of $\sum_{i=1}^N |v_i|^2$ among extremal vector systems.
\end{proof}

Finally, we prove a stability version of the estimate for the channel capacity.

\begin{proof}[Proof of Theorem~\ref{thm5}]
  Assume that $(v_i)_1^N$ is an extremal vector system provided by Theorem~\ref{thm3}. Let $A = \sum_{i=1}^N v_i \otimes v_i$ be the associated frame operator. As before,
\begin{equation}\label{ahs}
  \| A \|_{HS}^2 \geq \frac{\left(\sum_{i=1}^N |v_i|^2\right)^2}{d} \,.
\end{equation}
Let
\[
\mu = \min_k  \frac{|v_k|^2}{\sigma^2 + \sum_{l \neq k} \la v_k, v_l \ra^2}
\]
be the quantity for which we have to provide an upper bound. Then
\[
|v_k|^2  \geq \mu \sigma^2 + \mu \sum_{l \neq k} \la v_k, v_l \ra^2 = \mu \sigma^2 + \mu \sum_{l=1}^N \la v_k, v_l \ra^2 - \mu |v_k|^4
\]
holds for every $k$. By summing over $k$,
\begin{equation}\label{eqN}
  \sum_{k=1}^N |v_k|^2 \geq N \mu  \sigma^2 + \mu \| A \|_{HS}^2 - \mu \sum_{k=1}^N |v_k|^4 .
\end{equation}
Introduce $R =  \sum_{k=1}^N |v_k|^2$. By Theorem~\ref{thm3},
\begin{equation}\label{Rint}
  N c_1 \leq R \leq N c_1 + d \frac {c_1^2 - \sigma^2}{c_1^2 - d \sigma^2 } (c_2 - c_1)
\end{equation}
and
\[
\sum_{k=1}^N |v_k|^4  \leq N c_1^2  + d (c_2^2 - c_1^2) \frac {c_1^2 - \sigma^2}{c_1^2 - d \sigma^2 } \, .
\]
Therefore, \eqref{ahs} and  \eqref{eqN} lead to
\[
R \geq \mu \left( N \sigma^2 + \frac {R^2}{d}   - N c_1^2  - d (c_2^2 - c_1^2) \frac {c_1^2 - \sigma^2}{c_1^2 - d \sigma^2 }  \right).
\]
Since $R \geq N c_1$, the conditions $\sigma<c_1 / \sqrt{2d }$ and  $N > 2 d c_2^2 / c_1^2$ ensure that  the second term of the right-hand side is strictly positive. Then
\begin{equation}\label{muineq}
\mu \leq \frac{R}{N \sigma^2 + \frac{R^2}{d}   - N c_1^2  - d (c_2^2 - c_1^2) \frac{c_1^2 - \sigma^2}{c_1^2 - d \sigma^2 } }\,.
\end{equation}
In order to obtain an upper bound for $\mu$, we  maximize this quantity as a function of $R$ over the interval given by \eqref{Rint}. By a simple calculation one obtains that the conditions on $\sigma$ and $N$ imply that
\[
N^2 c_1^2 > N d (c_1^2 - \sigma^2) +  d^2 (c_2^2 - c_1^2) \frac {c_1^2 - \sigma^2}{c_1^2 - d \sigma^2 }.
\]
Therefore, \eqref{muineq} is decreasing over the whole interval defined by \eqref{Rint}.   Thus, its maximum value is attained at $R = N c_1$, which by \eqref{muineq} leads to the bound
\[
\mu \leq \frac{c_1}{\sigma^2 + (\frac N d - 1)c_1^2  - \frac d N  (c_2^2 - c_1^2) \frac{c_1^2 - \sigma^2 } { c_1^2 - d \sigma^2 }}\,.
\qedhere
\]
\end{proof}

\section{Acknowledgement}
We are grateful to Prof. Ed Saff for his valuable advices and to the anonymous referees for several suggestions improving the presentation of the results, and for the simplification of the proof of Theorem~\ref{thm1}.
This work is supported by the  Technical Cooperation Project of HUAWEI.

\end{document}